\documentclass[12pt,a4paper]{amsart}
\usepackage{tgtermes}

\usepackage[margin=1in,a4paper]{geometry}
\usepackage[foot]{amsaddr}
\usepackage{enumitem}
\usepackage{lineno}
\usepackage{hyperref}
\hypersetup{
    colorlinks=true,
    citecolor=blue,
    linkcolor=blue,
    filecolor=blue,
    urlcolor=blue,
    pdftitle={A proof of the cycle double cover conjecture by OpenAI: An exposition},
    pdfauthor={Sang-il Oum}
}
\usepackage{keytheorems}
\usepackage{zref-clever}
\zcsetup{cap=true}
\newkeytheorem{theorem}[name=Theorem,refname={Theorem,Theorems}, Refname={Theorem,Theorems}]
\newkeytheorem{lemma,proposition,corollary,problem}[sibling=theorem]
\newkeytheorem{conjecture}[Refname={Conjecture,Conjectures},sibling=theorem]
\newcommand\abs[1]{\left\lvert #1\right\rvert}

\begin{document}
\title{A proof of the cycle double cover conjecture by OpenAI: An~exposition}
\author{Sang-il Oum}
\address{Discrete Mathematics Group, Institute for Basic Science (IBS), Daejeon, South Korea}
\email{sangil@ibs.re.kr}
\thanks{Supported by the Institute for Basic Science (IBS-R029-C1).}
\date\today
\begin{abstract}
    The cycle double cover conjecture states that every bridgeless graph 
    has a list of cycles such that every edge is in exactly two of them. 
    In July 2026, OpenAI announced a proof. 
    This exposition presents the proof with slight modifications
    intended to make it more accessible. 
\end{abstract}
\maketitle
\section{Introduction}

A graph is \emph{Eulerian} if every vertex has even degree.
A \emph{cycle} of a graph is a $2$-regular connected subgraph. 

A \emph{cycle double cover (CDC)} is a list of cycles such that 
every edge is in exactly two of them.
Since every Eulerian subgraph can be decomposed into pairwise edge-disjoint cycles, for the existence of a cycle double cover, it is equivalent to define a cycle double cover as a list of Eulerian subgraphs such that every edge is in exactly two of them.

A \emph{bridge} of a graph is an edge whose deletion increases the number of components.
Observe that a bridge is not in any cycle. Thus to have a cycle double cover, it is necessary to be bridgeless.

This note is to discuss a proof of the cycle double cover conjecture, announced by OpenAI~\cite{openai2026}.

\begin{theorem}[store=cdc,label={Cycle Double Cover Theorem}]\label{thm:cdc}
    Every bridgeless graph has a cycle double cover.
\end{theorem}

According to Zhang~\cite{Zhang2012,Zhang2016a}, 
this conjecture was independently stated by Seymour~\cite{Seymour1979c} for all bridgeless graphs  
and Szekeres~\cite{Szekeres1973} for all bridgeless cubic graphs.
An equivalent version was stated by Itai and Rodeh~\cite{IR1978}.

One of the motivations to study the cycle double cover conjecture was the following stronger conjecture on embeddings of graphs. 
An \emph{embedding} of a graph on a surface is a drawing of $G$ on the surface such that no two edges cross. An embedding is called a \emph{$2$-cell embedding} if every face is homeomorphic to an open disk. 
A \emph{strong} embedding of a graph $G$ on a surface $\Sigma$ is a $2$-cell embedding such that every face boundary is a cycle of~$G$.
\begin{conjecture}[Strong Embedding Conjecture]\label{conj:strong}
    Every $2$-connected graph has a strong embedding on a surface.
\end{conjecture}

\zcref{conj:strong} implies \zcref{thm:cdc}, just by first reducing to $2$-connected graphs and taking all face boundaries.
\zcref{conj:strong} remains open for general graphs
and 
\zcref{thm:cdc} proves \zcref{conj:strong} for cubic graphs, see \cite{Jaeger1985a}.
For more about the cycle double cover conjecture, survey papers and books were 
written by Jaeger~\cite{Jaeger1985a}, Zhang~\cite{Zhang2012,Zhang2016a,Zhang1997}, and Chan~\cite{Chan2009}.

The proof presented in this note is due to GPT 5.6 announced by OpenAI \cite{openai2026} in July 2026.
The proof was modified mainly in two parts.
\begin{itemize}
    \item The choice of two values of $\mathbb{F}_2^3$ at each edge is now symmetric. We do not choose a particular edge at each vertex.
    \item The linear algebra proof is done without using dual vector spaces and is based on an elementary fact that the column space of a matrix 
    is the orthogonal complement of its left null space.
\end{itemize}

Our goal is to present the whole proof in mostly self-contained manner,
which would be suitable to teach the proof to advanced undergraduate students.
We will only assume two theorems: Fleischner's splitting lemma and a theorem on packing spanning trees. 
We allow parallel edges and loops in graphs.

The last section will discuss more conjectures related to the cycle double cover conjecture.

\bigskip
\noindent\textbf{Statement of AI use.}
Parts of the revision were carried out with GPT 5.6 under the author's guidance.

\section{Splitting lemma}
For an integer $k$, 
a graph~$G$ with at least two vertices is \emph{$k$-edge-connected} 
if for every $X\subseteq V(G)$ with $0<\abs{X}<\abs{V(G)}$, 
$G$ has at least $k$ edges having one end in $X$ and the other end in~$V(G)\setminus X$.
We omit the proof of the following lemma.
\begin{lemma}[Fleischner's splitting lemma~\cite[Lemma~III.26]{Fleischner1990}]
    \label{lem:fleischner-splitting}
    Let $G$ be a $2$-edge-connected graph, let $v$ be a vertex of degree at least
    four, and let $e_1,e_2,e_3$ be distinct edges incident with $v$. For each $i$, let
    $v_i$ be the other end of $e_i$. If $G-\{e_1,e_2,e_3\}$ is connected, then at
    least one of
    \[
        (G-\{e_1,e_2\})+v_1v_2
        \quad\text{and}\quad
        (G-\{e_2,e_3\})+v_2v_3
    \]
    is $2$-edge-connected.
\end{lemma}

If $G$ is $3$-edge-connected and $d(v)\ge4$, then the three edges in
\zcref{lem:fleischner-splitting} can always be chosen. If $G-v$ is connected, a
spanning tree of $G-v$ together with one edge incident with $v$ is a spanning tree of
$G$ that omits at least three edges incident with $v$. If $G-v$ is disconnected,
every component of $G-v$ sends at least three edges to $v$. Taking a spanning tree in
each component and one edge from that component to $v$ again gives a spanning tree of
$G$ that omits at least three edges incident with $v$. Any three such omitted edges
have the required property.
\section{Minimum counterexample to the cycle double cover conjecture}
We say a graph is \emph{cubic} if every vertex has degree $3$.
The following reduction is well known, see Zhang~\cite{Zhang2012}.
\begin{proposition}\label{prop:min}
    The cycle double cover conjecture is true if and only if it is true for 
    cubic $3$-edge-connected graphs.
\end{proposition}
\begin{proof}
    Suppose the cycle double cover conjecture is false, and let $G$ be a counterexample with
    $\abs{E(G)}$ minimum.    
    Deleting a loop and covering it twice shows that $G$ has no
    loops, and considering components shows that $G$ is connected.

    We first show that $G$ is $3$-edge-connected. Otherwise it has a $2$-edge-cut
    $\{e_1,e_2\}$. Contract~$e_1$, and let $w$ be the resulting vertex. Contraction
    cannot create a bridge, so $G/e_1$ is a smaller bridgeless graph and hence has a
    cycle double cover $\mathcal{C}$ by minimality.

    Consider the two shores of $G-\{e_1,e_2\}$. If $e_2$ is not a loop in $G/e_1$,
    a cycle of $\mathcal{C}$ containing~$e_2$ passes through~$w$ from one shore to
    the other, and a cycle not containing $e_2$ cannot do so. Split~$w$ back into
    the ends of $e_1$ and insert $e_1$ into the two cycle occurrences containing
    $e_2$; all other cycles lift unchanged. If $e_2$ is a loop in $G/e_1$, replace
    its two occurrences instead by two copies of the $2$-cycle formed by $e_1$ and
    $e_2$. In either case we obtain a cycle double cover of $G$, a contradiction.
    Thus $G$ is $3$-edge-connected.

    Suppose that $d(v)\ge4$ for some vertex $v$. Choose $e_1,e_2,e_3$ as above and
    apply \zcref{lem:fleischner-splitting}. Let~$H$ be the resulting
    $2$-edge-connected graph with $\abs{E(H)}<\abs{E(G)}$, so by minimality it has a
    cycle double cover. Replacing each occurrence of its new edge by the corresponding
    two-edge path through $v$, and decomposing closed trails into cycles if necessary,
    gives a cycle double cover of $G$, a contradiction. Hence $G$ is cubic.
\end{proof}

\section{Disjoint spanning trees}

\begin{theorem}[Tutte~\cite{Tutte1961a} and Nash-Williams~\cite{NashWilliams1961}]\label{thm:nashwilliams}
    A graph $G$ has $t$ edge-disjoint spanning trees if and only if
    for every partition $V_1,\ldots,V_m$ of $V(G)$ with nonempty parts,
    there are at least $t(m-1)$ edges joining distinct parts.
\end{theorem}
One of the easiest ways to prove \zcref{thm:nashwilliams} is to use the matroid union theorem.
There are purely graph-theoretic proofs known \cite{Kaiser2012}.
We omit its proof.

The following corollary of \zcref{thm:nashwilliams} is due to Kundu~\cite{Kundu1974}.
\begin{corollary}\label{cor:disjointspanningtree}
    If $G$ is $2k$-edge-connected, then
    $G$ has $k$ edge-disjoint spanning trees.
\end{corollary}
\begin{proof}
    Let $V_1,\ldots,V_m$ be a partition
    of $V(G)$ with nonempty parts.
    Each set $V_i$ has at least $2k$ edges in its edge boundary.
    When summing these for all $i$, we count each crossing edge twice.
    Therefore, the number of edges joining distinct parts of $V_1,\ldots,V_m$ is at least $km$,
    which is at least $k(m-1)$.
\end{proof}

Jaeger~\cite{Jaeger1976,Jaeger1979} and Kilpatrick~\cite{Kilpatrick1975} independently proved that
every $2$-edge-connected graph has a nowhere-zero $\mathbb{Z}_2^3$-flow
by using \zcref{thm:nashwilliams}.
Here I would like to write its main idea, without mentioning the nowhere-zero $8$-flows.

Both Jaeger~\cite[Proposition 8(a)]{Jaeger1979}
and Kilpatrick~\cite[Proposition 25.2]{Kilpatrick1975}
proved the following lemma by using the matroid union theorem of Edmonds~\cite{Edmonds1965b} directly. 
Later Jaeger~\cite{Jaeger1988} presented the following simpler proof based on \zcref{cor:disjointspanningtree}.

\begin{lemma}\label{lem:three-trees}
    Let $G$ be a $3$-edge-connected graph.
    Then there are three spanning trees $T_1$, $T_2$, $T_3$ such that
    $E(T_1)\cap E(T_2)\cap E(T_3)=\emptyset$.
\end{lemma}
\begin{proof}
    Let $G'$ be a graph obtained from $G$ by adding a parallel copy of every edge so that $\abs{E(G')}=2\abs{E(G)}$.
    Then $G'$ is $6$-edge-connected.
    By \zcref{cor:disjointspanningtree},
    $G'$ has three edge-disjoint spanning trees $T_1'$, $T_2'$, $T_3'$.
    Those three spanning trees correspond to spanning trees $T_1$, $T_2$, $T_3$ of $G$
    such that no edge belongs to all three of them at the same time.
\end{proof}

For a vertex $v$ of a graph $G$, let
$\delta(v)=\{e\in E(G):v\text{ is an end of }e\}$.
\begin{lemma}\label{lem:tree-to-eulerian}
    Let $T$ be a spanning tree of a graph $G$.
    Then there is a set $F$ of edges of $G$
    such that $E(G)\setminus E(T)\subseteq F$
    and $\abs{F\cap\delta(v)}$ is even for every vertex $v$ of $G$.
\end{lemma}
\begin{proof}
    For every $e\in E(G)\setminus E(T)$,
    there is a cycle $C_e$ in $T+e$. Define
    \[
        F=\mathop{\triangle}_{e\in E(G)\setminus E(T)} E(C_e),
    \]
    which is the set of all edges in an odd number of
    $C_e$'s.
    Now the conclusion follows.
\end{proof}

For an abelian group $\Gamma$, a \emph{nowhere-zero $\Gamma$-flow} of a graph $G$ 
is a function $\phi:E(\vec G)\to\Gamma$ for an orientation $\vec G$ of $G$ 
such that 
$\sum_{e\in \delta^+(v)}\phi(e)=\sum_{e\in \delta^-(v)}\phi(e)$ for every vertex $v$ of $G$
and $\phi(e)\neq0$ for all edges $e$ of $G$.
If $x=-x$ for every $x\in\Gamma$, then there is no need to specify an orientation,
as in our setting.
\begin{theorem}[Jaeger \cite{Jaeger1976,Jaeger1979}]\label{thm:jaeger-8flow}
    Every $2$-edge-connected graph has a nowhere-zero $\mathbb{Z}_2^3$-flow.
\end{theorem}
We only need its special case.
\begin{lemma}[Jaeger \cite{Jaeger1976,Jaeger1979}; weaker]\label{lem:jaeger-flow}
    Every $3$-edge-connected graph has a nowhere-zero $\mathbb{Z}_2^3$-flow.
\end{lemma}
\begin{proof}
By
    \zcref{lem:three-trees}, choose spanning trees $T_1,T_2,T_3$ such that
    $E(T_1)\cap E(T_2)\cap E(T_3)=\emptyset$. For each $i\in\{1,2,3\}$,
    \zcref{lem:tree-to-eulerian} gives a set $F_i\subseteq E(G)$ containing
    $E(G)\setminus E(T_i)$ and having even degree at every vertex. Let $\phi_i$ be
    the characteristic function of $F_i$ and define
    \[
        \phi:E(G)\longrightarrow\mathbb{Z}_2^3,
        \qquad
        \phi(e)=
        \begin{pmatrix}
            \phi_1(e)\\
            \phi_2(e)\\
            \phi_3(e)
        \end{pmatrix}.
    \]
    Each coordinate sums to zero at every vertex. Moreover, every edge $e$ is
    omitted by some $T_i$ and hence belongs to $F_i$. Thus $\phi(e)\ne0$, so
    $\phi$ is a nowhere-zero $\mathbb{Z}_2^3$-flow.
\end{proof}

\section{Building a cycle double cover}

\begin{lemma}[Two-element edge labels]\label{lem:two-element-edge-labels}
    Let $G$ be a loopless graph and let $\Gamma$ be a finite set. Suppose that every edge $e$ is
    assigned a two-element set $P_e\subseteq\Gamma$ such that
    \[
        \abs{\{e\in\delta(v):s\in P_e\}} 
        \text{ is even}
        \qquad \text{ for all }v\in V(G) \text{ and } s\in\Gamma.
    \]
    Then $G$ has a cycle double cover.
\end{lemma}
\begin{proof}
    For each $s\in\Gamma$, let
    \[
        H_s=\{e\in E(G):s\in P_e\}.
    \]
    Every vertex has even degree in $H_s$, so $H_s$ is Eulerian. Every edge belongs to exactly two of the graphs $H_s$, since
    $P_e$ has two elements. Thus all the $H_s$ form a cycle double cover of $G$.
\end{proof}

Remember that for a subspace $W'$ of a vector space $W$ and an element $x$ of $W$, 
we have $x+W'=\{x+v:v\in W'\}$.
For a vector $x$ in $W$, let $\langle x\rangle$ denote the span of $x$.
Let $\mathbb F_2=\{0,1\}$ be a field with two elements.

\begin{lemma}[Flow lifting]\label{lem:flow-lifting}
    Let $G$ be a cubic loopless graph. 
    Let $\Gamma=\mathbb F_2^3$.
    Let $\phi:E(G)\to\Gamma$ be a function such that 
    $\phi(e)\neq0$ for every edge $e$ of~$G$ 
    and $\sum_{e\in \delta(v)}\phi(e)=0$ for every vertex~$v$ of~$G$.
    If there are $t_v\in \Gamma$ for all $v\in V(G)$ such that 
    for every edge $e=uv$ and all
    $f_u\in\delta(u)\setminus\{e\}$ and $f_v\in\delta(v)\setminus\{e\}$,
    \[ t_u + t_v \in   (\phi(f_u)+\phi(f_v)) + \langle \phi(e)\rangle, \] 
    then $G$ has a cycle double cover.
\end{lemma}
\begin{proof}
    Note that if $e,f,g\in\delta(v)$ are distinct edges incident with $v$, then 
    $\phi(e)+\phi(f)+\phi(g)=0$ and therefore 
    $\phi(f)+\langle \phi(e)\rangle = \phi(g)+\langle \phi(e)\rangle$. 
    Thus the choice of $f_u$ or $f_v$ does not change the value.

    For each edge $e=uv$ of $G$, 
    let $P_e= t_u+ \phi (f_u)+\langle \phi(e)\rangle$. 
    By the choice of $t_v$, we have  
    $P_e=t_v +\phi(f_v)+\langle \phi(e)\rangle$
    and therefore $P_e$ is well defined.
    Furthermore, $\abs{P_e}=2$.

    Observe that 
    for a vertex $v$ with three incident edges $e$, $f$, $g$, 
    \begin{align*}
        P_e&=\{ t_v + \phi(f), t_v+\phi(g)\}, \\
        P_f&=\{ t_v + \phi(e), t_v+\phi(g)\}, \\
        P_g&=\{ t_v +\phi(e), t_v +\phi(f)\}.
    \end{align*}
    Now, we apply \zcref{lem:two-element-edge-labels}.
\end{proof}

\section{Some linear algebra}

For an $m\times n$ matrix $A$ over a field $\mathbb F$, 
we write $\mathcal C(A)=\{Ax\in \mathbb F^m: x\in \mathbb F^n\}$ for the \emph{column space}
and $\ker(A^{\mathsf T})=\{y\in \mathbb F^m: A^{\mathsf T}y=0\}$ for the \emph{left nullspace}.

\begin{lemma}[Column space and left nullspace]\label{lem:column-left-nullspace}
    Let $\mathbb{F}$ be a field and let $A$ be an $m\times n$ matrix over $\mathbb F$. With respect to
    the standard dot product on $\mathbb{F}^m$,
    \[
        \mathcal{C}(A)=\ker(A^{\mathsf T})^\perp.
    \]
\end{lemma}
\begin{proof}
    If $\mathbf{y}\in\ker(A^{\mathsf T})$, then $\mathbf{y}^{\mathsf T}A=0$, so
    $\mathbf{y}$ is orthogonal to every column of $A$. Thus
    $\mathcal{C}(A)\subseteq\ker(A^{\mathsf T})^\perp$. The two spaces have the same
    dimension, namely $\operatorname{rank}A$, and hence they are equal.
\end{proof}

\begin{proposition}[Elementary consistency criterion]\label{lem:elementary-consistency}
    Let $\mathbb{F}$ be a field, let $n\ge2$ be an integer, and put
    $W=\mathbb{F}^n$ with the standard dot product. Let $G$ be a graph.
    For every edge $e$, let $p_e$ be a nonzero vector in $W$
    and let $d_e\in W$.
    Then the following are equivalent.
    \begin{enumerate}[label=\textup{(\roman*)}]
        \item There are vectors $t_v\in W$ such that
        \[
            t_u+t_v\in d_e+\langle p_e\rangle
            \qquad\text{for all }e=uv.
        \]
        \item For every family $(h_e)_{e\in E(G)}$ of vectors in $W$, if 
        \[
            h_e\mathbin{\cdot}p_e=0\quad\text{for all $e\in E(G)$}
            \qquad
            \text{ and }
            \qquad
            \sum_{e\in\delta(v)}h_e=0\quad\text{for all $v\in V(G)$},
            \] 
        then 
        \(
            \sum_{e\in E(G)}h_e\mathbin{\cdot}d_e=0
        \).
    \end{enumerate}
\end{proposition}
\begin{proof}
    Suppose first that \textup{(i)} holds, and let $(h_e)_{e\in E(G)}$ satisfy the two
    hypotheses in \textup{(ii)}. For each edge $e=uv$, choose $\lambda_e\in\mathbb{F}$
    such that
    \[
        t_u+t_v=d_e+\lambda_e p_e.
    \]
    Since $h_e\mathbin{\cdot}p_e=0$, we have
    \[
        \sum_{e\in E(G)}h_e\mathbin{\cdot}d_e
        =\sum_{e=uv\in E(G)}h_e\mathbin{\cdot}(t_u+t_v)
        =\sum_{v\in V(G)}
          \left(\sum_{e\in\delta(v)}h_e\right)\mathbin{\cdot}t_v
        =0.
    \]
    Thus \textup{(ii)} holds.

    Conversely, suppose that \textup{(ii)} holds. For each edge $e$, choose a basis
    $r_{e,1},\ldots,r_{e,n-1}$ of
    $\langle p_e\rangle^\perp=\{h\in W:h\mathbin{\cdot}p_e=0\}$.
    Let $R_e$ be the $(n-1)\times n$ matrix whose $i$-th row is $r_{e,i}^{\mathsf T}$.
    Since $(\langle p_e\rangle^\perp)^\perp=\langle p_e\rangle$, condition \textup{(i)} is equivalent to
    \[  t_u + t_v - d_e \in \langle r_{e,1},r_{e,2},\ldots,r_{e,n-1}\rangle^\perp \qquad\text{for all }e=uv,\]
    which is equivalent to 
    \[
        R_e(t_u+t_v)=R_ed_e\qquad\text{for all }e=uv.
    \]

    Form the block matrix
    \[
        A\in\mathbb{F}^{(n-1)\abs{E(G)}\times n\abs{V(G)}}
    \]
    with block rows indexed by $e\in E(G)$ and block columns indexed by $v\in V(G)$, where
    \[
        A_{e,v}=\begin{cases}
            R_e&\text{if }e\in\delta(v),\\
            0&\text{otherwise}.
        \end{cases}
    \]
    If $\mathbf{t}=(t_v)_{v\in V(G)}$ and
    $\mathbf{b}=(R_ed_e)_{e\in E(G)}$, then condition \textup{(i)} is precisely the
    solvability of
    \[
        A\mathbf{t}=\mathbf{b}.
    \]

    By \zcref{lem:column-left-nullspace}, it remains to show that $\mathbf{b}$ is
    orthogonal to the left nullspace $\ker(A^{\mathsf T})$. Take
    $\mathbf{y}=(y_e)_{e\in E(G)}\in\ker(A^{\mathsf T})$, where
    $y_e\in\mathbb{F}^{n-1}$, and put
    \[
        h_e=R_e^{\mathsf T}y_e.
    \]
    Then $h_e\in \langle p_e\rangle^\perp$, and the block of $A^{\mathsf T}\mathbf{y}$ indexed by a
    vertex $v$ is $\sum_{e\in\delta(v)}h_e$. Hence
    $\sum_{e\in\delta(v)}h_e=0$ for every vertex $v$. By \textup{(ii)},
    \[
        \mathbf{y}^{\mathsf T}\mathbf{b}
        =\sum_{e\in E(G)}y_e^{\mathsf T}R_ed_e
        =\sum_{e\in E(G)}h_e\mathbin{\cdot}d_e
        =0.
    \]
    Thus $\mathbf{b}\in\ker(A^{\mathsf T})^\perp=\mathcal{C}(A)$. Equivalently,
    $A\mathbf{t}=\mathbf{b}$ has a solution, and hence \textup{(i)} holds.
\end{proof}
\section{Proof for the cubic $3$-edge-connected graphs}

\begin{proposition}\label{prop:cubic-cdc}
    Let $G$ be a cubic $3$-edge-connected graph.
    Then $G$ has a cycle double cover.
\end{proposition}

\begin{proof}[Proof of \zcref{prop:cubic-cdc}]
    Put $\Gamma=\mathbb{F}_2^3$.
    \zcref{lem:jaeger-flow} yields a nowhere-zero
    $\Gamma$-flow $\phi:E(G)\to\Gamma$.

    If $v$ is an end of an edge $e$, choose either edge
    $f\in\delta(v)\setminus\{e\}$ and put $c_{v,e}=\phi(f)$.
    The two choices differ by $\phi(e)$, so the coset
    $c_{v,e}+\langle\phi(e)\rangle$ is independent of the choice.

    For every edge $e=uv$, put $d_e=c_{u,e}+c_{v,e}$.
    By  \zcref{lem:flow-lifting}, it is enough to show that 
    there are $t_v\in\Gamma$ for all $v\in V(G)$ such that
    \[
        t_u+t_v\in d_e+\langle\phi(e)\rangle
        \qquad\text{for all }e=uv.
        \tag{1}\label{eq:compatibility}
    \]
We apply \zcref{lem:elementary-consistency} with
    $\mathbb{F}=\mathbb{F}_2$, $n=3$, $W=\Gamma$, and $p_e=\phi(e)$.
    Consider any family
    $(h_e)_{e\in E(G)}\in\Gamma^{E(G)}$ satisfying
    \[
        h_e\mathbin{\cdot}\phi(e)=0\quad\text{for all }e\in E(G),
        \qquad
        \sum_{e\in\delta(v)}h_e=0\quad\text{for all }v\in V(G).
        \tag{2}\label{eq:edge-dependency}
    \]
    It remains to prove that $\sum_e h_e\mathbin{\cdot}d_e=0$.

    Fix a vertex $v$ and write $\delta(v)=\{e_1,e_2,e_3\}$.
    The three values $\phi(e_i)$ are nonzero and pairwise distinct, and
    \[
        \phi(e_1)+\phi(e_2)+\phi(e_3)=0,
        \qquad h_{e_1}+h_{e_2}+h_{e_3}=0.
    \]
    Put $\lambda=h_{e_2}\mathbin{\cdot}\phi(e_1)$. Using
    $h_{e_i}\mathbin{\cdot}\phi(e_i)=0$ in the two displayed identities gives
    \begin{align*}
        h_{e_1}\mathbin{\cdot}(\phi(e_1)+\phi(e_2)+\phi(e_3))&=h_{e_1}\mathbin{\cdot}\phi(e_2) + h_{e_1}\mathbin{\cdot}\phi(e_3)=0, \\ 
        h_{e_2}\mathbin{\cdot}(\phi(e_1)+\phi(e_2)+\phi(e_3))&=h_{e_2}\mathbin{\cdot}\phi(e_1) + h_{e_2}\mathbin{\cdot}\phi(e_3)=0, \\ 
        h_{e_3}\mathbin{\cdot}(\phi(e_1)+\phi(e_2)+\phi(e_3))&=h_{e_3}\mathbin{\cdot}\phi(e_1) + h_{e_3}\mathbin{\cdot}\phi(e_2)=0, \\ 
        (h_{e_1}+h_{e_2}+h_{e_3})\mathbin{\cdot}\phi(e_1)
        &= h_{e_2}\mathbin{\cdot}\phi(e_1) + h_{e_3}\mathbin{\cdot}\phi(e_1)=0, \\ 
        (h_{e_1}+h_{e_2}+h_{e_3})\mathbin{\cdot}\phi(e_2)&=h_{e_1}\mathbin{\cdot}\phi(e_2) + h_{e_3}\mathbin{\cdot}\phi(e_2)=0, \\ 
        (h_{e_1}+h_{e_2}+h_{e_3})\mathbin{\cdot}\phi(e_3)&=h_{e_1}\mathbin{\cdot}\phi(e_3) + h_{e_2}\mathbin{\cdot}\phi(e_3)=0, 
    \end{align*}
    and therefore 
    \[
        \lambda = h_{e_2}\mathbin{\cdot}\phi(e_1) 
        = h_{e_2}\mathbin{\cdot}\phi(e_3) 
        = h_{e_1}\mathbin{\cdot}\phi(e_3)
        = h_{e_1}\mathbin{\cdot}\phi(e_2)
        = h_{e_3}\mathbin{\cdot}\phi(e_2) 
        = h_{e_3}\mathbin{\cdot}\phi(e_1).
    \]
    If $\lambda=1$, these equalities show that all three $h_{e_i}$ are nonzero.
    If $\lambda=0$, all three lie in the one-dimensional space
    $\langle\phi(e_1),\phi(e_2),\phi(e_3)\rangle^\perp$, and their sum is zero, so an even
    number of them are nonzero. Thus
    \[
        \lambda=\sum_{e\in\delta(v)}\mathbf{1}_{h_e\ne0}.
    \]
    Here, $\mathbf{1}_{h_e\ne0}$ is a function that has value $1$ if $h_e\neq0$ and $0$ otherwise.
    Because changing the choice of $c_{v,e}$ adds $\phi(e)$ and
    $h_e\mathbin{\cdot}\phi(e)=0$, we may compute its dot product using either choice.
    Hence
    \[
        \sum_{e\in\delta(v)}h_e\mathbin{\cdot}c_{v,e}
        =h_{e_1}\mathbin{\cdot}\phi(e_2)
         +h_{e_2}\mathbin{\cdot}\phi(e_1)
         +h_{e_3}\mathbin{\cdot}\phi(e_1)
        =\lambda
        =\sum_{e\in\delta(v)}\mathbf{1}_{h_e\ne0}.
        \tag{3}\label{eq:local-sum}
    \]

    Since $d_e=c_{u,e}+c_{v,e}$, summing \eqref{eq:local-sum} over all vertices gives
    \[
        \sum_{e\in E(G)}h_e\mathbin{\cdot}d_e
        =\sum_{v\in V(G)}\sum_{e\in\delta(v)}h_e\mathbin{\cdot}c_{v,e}
        =\sum_{v\in V(G)}\sum_{e\in\delta(v)}\mathbf{1}_{h_e\ne0}=0.
    \]
    The last equality holds in $\mathbb{F}_2$, because every edge with $h_e\ne0$
    is counted at both ends. By \zcref{lem:elementary-consistency}, the required
    $(t_v)_{v\in V(G)}$ exist. This proves the claim.
\end{proof}

\section{Completing the proof}

\getkeytheorem{cdc}

\begin{proof}
    By \zcref{prop:min}, it is enough to prove it for cubic $3$-edge-connected graphs.
    \zcref{prop:cubic-cdc} gives the final contradiction.
\end{proof}

\section{More conjectures}
\subsection{Small cycle double covers}

In 1980, Bondy~\cite{Bondy1990a} proposed the following strengthening of the cycle double cover
conjecture.
\begin{conjecture}[Bondy~\cite{Bondy1990a}]\label{conj:small-cdc}
    Every simple $2$-edge-connected graph on $n$ vertices has a cycle double
    cover consisting of at most $n-1$ cycles.
\end{conjecture}
The bound would be best possible. Indeed, in a cycle double cover of $K_n$,
the $2(n-1)$ occurrences of edges incident with a fixed vertex must be paired
into $n-1$ occurrences of cycles through that vertex.
\zcref{conj:small-cdc} does not follow immediately from the
method of this paper.

In the same paper,
Bondy~\cite{Bondy1990a} also conjectured that every simple $2$-connected cubic graph on~$n$
vertices other than~$K_4$ has a cycle double cover consisting of at most
$n/2$ cycles. This now follows from \zcref{thm:cdc}
because Lai, Yu, and Zhang~\cite{LYZ1994} proved that it is true if the graph has a cycle double cover.
\begin{corollary}[Lai, Yu, and Zhang~\cite{LYZ1994}]\label{cor:small-cubic-cdc}
    Every simple $2$-connected cubic graph on $n$ vertices other than $K_4$
    has a cycle double cover consisting of at most $n/2$ cycles.
\end{corollary}

The exception is necessary: the three Hamilton cycles of $K_4$ form a cycle
double cover, whereas no two cycles can cover its six edges twice.

\subsection{$5$-cycle double cover conjecture}

A \emph{$k$-cycle double cover ($k$-CDC)} is a collection of at most $k$ Eulerian subgraphs such that every edge is in exactly two of them.

The proof of \zcref{thm:cdc} also yields the following.
\begin{theorem}\label{thm:8-cycle-double-cover}
    Every bridgeless graph has an $8$-cycle double cover.
\end{theorem}

Since the Petersen graph is not $3$-edge-colorable,
not all graphs have a $4$-cycle double cover by the following theorem.

\begin{theorem}[Jaeger~\cite{Jaeger1985a}]
    The following are equivalent for loopless cubic graphs $G$.
    \begin{enumerate}[label=\rm(\roman*)]
        \item $G$ is $3$-edge-colorable.
        \item $G$ has a $3$-cycle double cover.
        \item $G$ has a $4$-cycle double cover.
    \end{enumerate}
\end{theorem}
\begin{proof}
    (iii)$\to$(ii): If $G$ has a $4$-cycle double cover $C_1$, $C_2$, $C_3$, $C_4$, 
    then 
    $C_1\triangle C_2$, $C_1\triangle C_3$, $C_1\triangle C_4$ is a $3$-cycle double cover.

    (ii)$\to$(i): Let $C_1,C_2,C_3$ be a $3$-cycle double cover. Color each edge by
    the unique index $i\in\{1,2,3\}$ for which $e\notin C_i$. At every vertex, each
    $C_i$ has degree two: its degree is even and the sum of these three degrees is
    six. Thus each $C_i$ omits exactly one of the three incident edges. Since every
    edge is omitted by exactly one $C_i$, the three incident edges receive distinct
    colors. Hence this is a proper $3$-edge-coloring.

    (i)$\to$(iii): Let $M_1,M_2,M_3$ be the three color classes of a proper
    $3$-edge-coloring. They are perfect matchings forming a partition of $E(G)$.
    The Eulerian subgraphs
    \[
        M_1\triangle M_2,\qquad M_2\triangle M_3,\qquad M_1\triangle M_3
    \]
    form a $3$-cycle double cover: an edge in $M_i$ belongs to precisely the two
    symmetric differences containing $M_i$. This is also a $4$-cycle double cover.
\end{proof}

There is a conjecture that one can reduce $8$ to $5$. 

\begin{conjecture}[Celmins~\cite{Celmins1985}, Preissmann~\cite{Preissmann1981}]
\label{conj:5-cdc}
    Every bridgeless graph has a $5$-cycle double cover.
\end{conjecture}
\subsection{Orientable cycle double cover conjecture}
An \emph{orientable cycle double cover} of a graph~$G$ is a list of Eulerian subgraphs, each equipped with a directed Eulerian orientation,
such that 
every edge of $G$ is used exactly once in each direction.
An \emph{orientable $k$-cycle double cover} of a graph~$G$ is 
an orientable cycle double cover with at most $k$ Eulerian subgraphs.

\begin{conjecture}[Archdeacon~\cite{Archdeacon1984} and Jaeger~\cite{Jaeger1988}]
    Every bridgeless graph admits an orientable $5$-cycle double cover.
\end{conjecture}
Actually, this conjecture would imply the $5$-flow conjecture of Tutte.
\begin{lemma}[Lemma 13.1.5 in \cite{Zhang2012}]
    If a graph $G$ admits an orientable $k$-cycle double cover,
    then it has a nowhere-zero $\mathbb{Z}_k$-flow.
\end{lemma}
\begin{proof}
    Pad the list with empty directed Eulerian subgraphs, and write it as
    $C_1,\ldots,C_k$. Fix an arbitrary orientation of $G$. For each $i$, let
    $\chi_i$ be the signed characteristic function of $C_i$ with respect to this
    orientation. Since $C_i$ is directed Eulerian, $\chi_i$ is an integer flow.
    Therefore
    \[
        \phi=\sum_{i=1}^k i\chi_i\pmod{k}
    \]
    is a $\mathbb{Z}_k$-flow. For each edge $e$, there are $i\ne j$ such that $C_i$
    and $C_j$ use $e$ in opposite directions. Hence $\phi(e)=\pm(i-j)\ne0$ in
    $\mathbb{Z}_k$, so $\phi$ is nowhere-zero.
\end{proof}
\subsection{Berge-Fulkerson conjecture}
An \emph{$m$-cycle $k$-cover} is a list of $m$ Eulerian subgraphs such that every edge belongs to exactly $k$ of them. 

\begin{theorem}[Bermond, Jackson, and Jaeger~\cite{BJJ1983}]
    Every bridgeless graph has a $7$-cycle $4$-cover.
\end{theorem}
\begin{proof}
    It suffices to prove the claim for each connected component of $G$. Thus, after
    covering any loops twice, we may assume that $G$ is $2$-edge-connected. By
    \zcref{thm:jaeger-8flow}, let $C_1,C_2,C_3$ be the supports of the three
    coordinates of a nowhere-zero $\mathbb{F}_2^3$-flow. Every edge belongs to at
    least one of these Eulerian subgraphs. The seven nonzero linear combinations
    \[
        C_1,\ C_2,\ C_3,\ C_1\triangle C_2,\ C_2\triangle C_3,\
        C_1\triangle C_3,\ C_1\triangle C_2\triangle C_3
    \]
    form a $7$-cycle $4$-cover, because every nonzero vector of $\mathbb{F}_2^3$ has
    dot product $1$ with exactly four of the seven nonzero vectors of
    $\mathbb{F}_2^3$.
\end{proof}
By using the $6$-flow theorem of Seymour~\cite{Seymour1981}, 
Fan~\cite{Fan1992} showed the following.
\begin{theorem}[Fan~\cite{Fan1992}]
    Every bridgeless graph admits a $10$-cycle $6$-cover.
\end{theorem}

\begin{conjecture}[Berge and Fulkerson~\cite{Fulkerson1971}]
    Every bridgeless cubic graph admits $6$ perfect matchings such that 
    every edge is in exactly two of them.
\end{conjecture}
If we take the complement of each perfect matching, we deduce the following equivalent form of the conjecture of Berge and Fulkerson.
\begin{conjecture}[Berge and Fulkerson~\cite{Fulkerson1971}]
    Every bridgeless cubic graph admits a 
    $6$-cycle $4$-cover.
\end{conjecture}

\subsection{Matroids}

A matroid is \emph{binary} if it is representable over $\mathbb{F}_2$, and
it is \emph{regular} if it is representable over every field. A
\emph{circuit} of a matroid is a minimal dependent set. The dual
matroid~$M^*$ has as its bases the complements of the bases of~$M$, and a
\emph{cocircuit} of~$M$ is a circuit of~$M^*$. A \emph{cycle} of a binary
matroid is a disjoint union of circuits, with the empty set allowed. A
\emph{cocycle} is a cycle of the dual matroid, or equivalently a disjoint
union of cocircuits. A matroid is \emph{graphic} if it is the cycle matroid
of a graph. In the cycle matroid of a graph, the cycles are precisely the
edge sets of Eulerian subgraphs.

A \emph{$k$-cycle double cover} of a matroid is a family of at most
$k$ cycles such that every element belongs to exactly two members of the
family. A \emph{cycle double cover} is such a family with no prescribed
bound on its size.
Let $F_7^*$ denote the dual of the Fano matroid. Jamshy and
Tarsi~\cite{JT1989} proved that the cycle double cover conjecture is
equivalent to the apparently stronger statement below.
By \zcref{thm:cdc}, we can state it as a theorem.
\begin{theorem}[Jamshy and Tarsi~\cite{JT1989}]\label{thm:f7minor}
    Every binary matroid without coloops and with no $F_7^*$-minor has a
    cycle double cover.
\end{theorem}

Since every regular matroid is binary and has no $F_7^*$-minor, this
theorem implies that every regular matroid without coloops has a cycle
double cover.

By \zcref{thm:8-cycle-double-cover}, 
every graphic matroid without coloops has an $8$-cycle double cover.
Furthermore, \zcref{conj:5-cdc} has the following equivalent matroid form.
\begin{conjecture}\label{con:5cdc-matroid}
    Every graphic matroid without coloops has a $5$-cycle double cover.
\end{conjecture}

One may ask whether \zcref{con:5cdc-matroid} extends from graphic matroids
to regular matroids, in analogy with \zcref{thm:f7minor}. It does not; in
fact, no uniform bound is possible for regular matroids. Linial, Meshulam,
and Tarsi proved that the minimum number of cocycles in a double cocycle
cover of $K_n$ is $n$ for $n\geq 5$~\cite[Theorem~3.1(iii) and
Section~6]{LMT1988}. Since the cycles of $M^*(K_n)$ are the cocycles of
$M(K_n)$, the minimum size of a cycle double cover of $M^*(K_n)$ is also
$n$. Moreover, $M^*(K_n)$ is regular and has no coloops. Thus the
$5$-cycle double cover conjecture for graphic matroids has no analogue
with any fixed number of cycles for cographic matroids.

\section*{Acknowledgements}
The author thanks Adrian Bondy for bringing his conjecture on small cycle double covers to the author's attention and Vahan Mkrtchyan for several remarks that helped improve the presentation.

\bibliographystyle{amsplain}
\providecommand{\bysame}{\leavevmode\hbox to3em{\hrulefill}\thinspace}
\providecommand{\MR}{\relax\ifhmode\unskip\space\fi MR }
\providecommand{\MRhref}[2]{\href{http://www.ams.org/mathscinet-getitem?mr=#1}{#2}
}
\providecommand{\href}[2]{#2}

\end{document}